\documentclass[12pt]{amsart}
\usepackage{graphics,amsmath,amsfonts,amssymb,amscd,mathrsfs}

\input{xy}
\xyoption{all}

\numberwithin{equation}{section}
\parskip=\smallskipamount

\newtheorem{theorem}{Theorem}[section]
\newtheorem{lemma}[theorem]{Lemma}

\newtheorem{proposition}[theorem]{Proposition}

\theoremstyle{definition}
\newtheorem{definition}[theorem]{Definition}
\newtheorem{example}[theorem]{Example}

\newtheorem{remark}[theorem]{Remark}

\newcommand\hra{\hookrightarrow}

\newcommand{\cO}{\mathscr{O}}

\newcommand{\B}{\mathbb{B}}
\newcommand{\C}{\mathbb{C}}
\newcommand{\D}{\mathbb{D}}

\newcommand{\N}{\mathbb{N}}
\newcommand{\Q}{\mathbb{Q}}
\newcommand{\Z}{\mathbb{Z}}

\newcommand{\R}{\mathbb{R}}

\newcommand\wt{\widetilde}
\newcommand{\Aut}{\mathop{{\rm Aut}}}
\newcommand\dist{\mathrm{dist}}

\def\bs{\backslash}

\begin{document}
\title[Proper holomorphic embeddings into Stein manifolds]
{Proper holomorphic embeddings into Stein manifolds with the density property}
\author{Rafael Andrist} 
\author{Franc Forstneri\v{c}}
\author{Tyson Ritter}
\author{Erlend Forn\ae ss Wold}
\address{R. Andrist, Bergische Universit\"at Wuppertal,
Fachbereich C - Mathematik und Naturwissenschaften, Gausstr.\ 20, D-42119 Wuppertal}
\email{rafael.andrist@math.uni-wuppertal.de}
\address{F.\ Forstneri\v c, Institute of Mathematics, Physics and Mechanics, 
University of Ljubljana, Jadranska 19, 1000 Ljubljana, Slovenia}
\email{franc.forstneric@fmf.uni-lj.si}
\address{T.\ Ritter, School of Mathematical Sciences, University of Adelaide, Adelaide SA 5005, Australia}
\email{tyson.ritter@adelaide.edu.au}
\address{E.\ F. Wold, Matematisk Institutt, Universitetet i Oslo,
Postboks 1053 Blindern, 0316 Oslo, Norway}
\email{erlendfw@math.uio.no}
%
%    General info
%
\subjclass[2010]{Primary 32E10, 32E20, 32E30, 32H02.  Secondary 32Q99.}
\keywords{Oka principle, Stein manifold, density property, Oka manifold, Lie group, holomorphic convexity}
\date{\today}

\begin{abstract}
We prove that a Stein manifold of dimension $d$ admits a proper holomorphic embedding into any Stein manifold of dimension at least $2d+1$ satisfying the holomorphic density property. This generalizes classical theorems of Remmert, Bishop and Narasimhan, pertaining to embeddings into complex Euclidean spaces, as well as several other recent results.
\end{abstract}
\maketitle

%
%
%   SECTION 1: INTRODUCTION
%
%
\section{Introduction}
\label{sec:Intro}
A complex manifold $X$ is said to satisfy the {\em density property} 
if the Lie algebra generated by all the $\C$-complete holomorphic vector fields on $X$ is dense in the Lie algebra of all holomorphic vector fields on $X$ in the compact-open topology. (See Varolin \cite{Varolin2001,Varolin2000}
or \cite[\S 4.10]{Forstneric2011}.) This condition  trivially holds on compact manifolds
where every vector field is complete, but is fairly restrictive on noncompact manifolds. It is especially interesting on Stein manifolds where it implies the Anders\'en-Lempert-Forstneri\v c-Rosay theorem on approximation of isotopies of injective holomorphic maps of Runge domains by holomorphic automorphisms. (See \cite{Andersen,Andersen-Lempert, Forstneric-Rosay} for the case $X=\C^n$ and Theorem 4.10.6 in \cite[p.\ 132]{Forstneric2011} or \cite[Appendix]{Ritter} for the general case.) Similarly one defines the {\em volume density property} of a Stein manifold $X$, endowed with a holomorphic volume form $\omega$, by considering the Lie algebra of all holomorphic vector fields on $X$ with vanishing $\omega$-divergence; their flows can be approximated by $\omega$-preserving automorphisms of $X$.

In this paper we prove the following main result.

\begin{theorem}
\label{th:main}
Let $X$ be a Stein manifold satisfying the density property or the volume density property. If $S$ is a Stein manifold and $2\dim S +1\le \dim X$, then any continuous map $f\colon S\to X$ is homotopic to a proper holomorphic embedding $F\colon S\hra X$. If in addition $K$ is a compact $\cO(S)$-convex set in $S$ 
such that $f$ is holomorphic on a neighborhood of $K$, and $S'$ is a closed complex subvariety of $S$
such that the restricted map $f|_{S'} \colon S'\hra X$ is a proper holomorphic embedding of $S'$ to $X$, 
then $F$ can be chosen to agree with $f$ on $S'$ and to approximate $f$ uniformly on $K$ as closely as desired.  
\end{theorem}

Here $\cO(S)$ denotes the algebra of all holomorphic functions on a complex manifold $S$, endowed with the compact-open topology. A compact set $K$ in $S$ is said to be {\em $\cO(S)$-convex} if for every point $x\in S\setminus K$ there exists a $g\in \cO(S)$ with $|g(x)| > \sup_K |g|$. 

Since a Stein manifold $X$ with the density property is an {\em  Oka manifold} (see \cite{Kaliman-Kut2008} or Theorem 5.5.18 in \cite[p.\ 206]{Forstneric2011}), it follows that every continuous map $S\to X$ from a Stein manifold $S$ is homotopic to a holomorphic map (see Theorem 5.4.4 in \cite[p.\ 193]{Forstneric2011}). Furthermore, the jet transversality theorem \cite[\S 7.8]{Forstneric2011} shows that a generic holomorphic map $S\to X$ is an injective immersion when $2\dim S +1\le \dim X$. 
The main new point of Theorem \ref{th:main} is that it gives {\em proper} holomorphic embeddings. This is a nontrivial addition since the usual Oka property for holomorphic maps need not imply the corresponding Oka property for {\em proper} holomorphic maps (see \cite[Example 1.3]{DF2010}). We do not know whether Theorem \ref{th:main} holds for every Stein Oka manifold $X$. 

In the special case when $S$ is a relatively compact, smoothly bounded, strongly pseudoconvex domain in another Stein manifold $\wt S$, Theorem \ref{th:main} holds  (except perhaps the interpolation condition) for every Stein manifold $X$ without assuming the density property; see Drinovec Drnov\v sek and Forstneri\v c \cite{DF2007,DF2010}. 

We actually prove the following more precise result, showing in particular that a proper holomorphic embedding $F\colon S\hra X$ in Theorem \ref{th:main} can be chosen to avoid a given compact holomorphically convex subset $L$ of the target manifold $X$. Comparable results for mappings to $X=\C^n$ were obtained by Forstneri\v c and Ritter in \cite{ForstnericRitter}. Theorem \ref{th:main2} is proved in \S \ref{sec:proof} after we prepare the necessary tools in \S  \ref{sec:mainlemma}.

\begin{theorem}
\label{th:main2}
Assume that $X$ is a Stein manifold of dimension $n$ satisfying the density property or the volume density property, $L$ is a compact $\cO(X)$-convex set in $X$, $S$ is a Stein manifold of dimension $d$ with $2d+1\le n$, $S'$ is a closed complex subvariety of $S$, $K$ is a compact $\cO(S)$-convex subset of $S$, and $f\colon S\to X$ is a continuous map such that 
\begin{itemize}
\item[\rm (a)]  $f$ is holomorphic on a neighborhood of $K$,
\item[\rm (b)]  the restriction $f|_{S'}\colon S'\to X$ is a proper holomorphic embedding, and 
\item[\rm (c)]  $f(S\setminus \mathring K) \subset X\setminus L$.
\end{itemize}
Then there exist a proper holomorphic embedding $F\colon S\hookrightarrow X$ and a homotopy of continuous maps $f_t\colon S \to X$ $(t\in [0,1])$, with $f_0=f$ and $f_1=F$, satisfying the following conditions for every $t\in [0,1]$:
\begin{itemize}
\item $f_t$ is holomorphic on a neighborhood of $K$ and is uniformly close to $f$ on $K$, 
\item $f_t|_{S'}=f|_{S'}$, and 
\item $f_t(S\setminus \mathring K)\subset X\setminus L$. 
\end{itemize}
\end{theorem}

Applying Theorem \ref{th:main2} to maps $S\to X$ whose images do not intersect the compact subset $L\subset X$, we see that the complement $X\setminus L$ of any compact $\cO(X)$-convex set in a Stein manifold $X$ with the density property enjoys the {\em basic Oka property with approximation} 
for maps from Stein manifolds $S$ satisfying $2\dim S+1\le \dim X$. 

Condition (c) on the initial map $f$ in Theorem \ref{th:main2} can be replaced by the weaker condition that 
$f(bK \cup (S'\setminus K)) \subset X\setminus L$. By topological reasons, a map $f$ satisfying the latter 
condition can be deformed to a map satisfying condition (c) by a homotopy  that is fixed on a neighborhood 
of $K$ and on $S'$ (see  \cite[Remark 2]{ForstnericRitter}).

Theorems \ref{th:main} and \ref{th:main2} generalize classical results of Remmert \cite{Remmert}, Bishop \cite{Bishop1961}, and Narasimhan \cite{Narasimhan1960}, concerning the existence of a proper holomorphic embedding $S \hra\C^n$ of any Stein manifold $S$ with $2\dim S +1\le n$. The corresponding result on interpolation of embeddings into $\C^n$ on 
closed complex subvarities of $S$ was proved by Acquistapace, Broglia and Tognoli \cite{Acquistapace}. A new proof and generalizations of these classical theorems of complex analysis were given recently by Forstneri\v c and Ritter in \cite{ForstnericRitter}. 

In the special case when $\dim S=1$, i.e., $S$ is an open Riemann surface, the existence of a proper holomorphic embedding in Theorem \ref{th:main} was proved recently by Andrist and Wold \cite{AndristWold}. They also constructed proper holomorphic immersions $S\to X$ when $\dim X=2$. Here we adapt their construction to the case $\dim S>1$. Our proof follows the general strategy used in Oka theory (see e.g\ Chapter 5 in \cite{Forstneric2011}), but with nontrivial additions to ensure that we obtain {\em proper} holomorphic embeddings. The main technical step in the proof is to extend a holomorphic map (by approximation) across a convex bump $B$, attached to a strongly pseudoconvex domain $A$ in a Stein manifold $S$, so that $B$ is mapped into the complement $X\setminus L$ of a given compact $\cO(X)$-convex set $L$  (see Lemma \ref{lem:main}). The latter property is used to obtain a proper limit map. It is not known whether such complements $X\setminus L$ are Oka manifolds when $X$ is Oka; this is an open problem even for the complement of a ball in $\C^n$.
Instead we use tools from the {\em Anders\'en-Lempert theory} concerning the approximation of certain isotopies of biholomorphic maps of Runge domains in $X$ by holomorphic automorphisms of $X$. 
It seems likely that the method developed in \cite{ForstnericRitter} could also be used to extend a holomorphic map across a convex bump; see Remark \ref{rem:anotherapproach} below. 

At this time we do not see how to prove the analogous result for proper holomorphic maps $S\to X$ when $\dim S<\dim X$ but $2\dim S+1>\dim X$ (cf.\ \cite{ForstnericRitter} for the case $X=\C^n$). The reason is that our method requires that the holomorphic map in an inductive step is an injective immersion on the attaching set $A\cap B$ of the bump $B$. We appeal to a general position argument to achieve this, which necessitates that $2\dim S+1\le \dim X$. 

A much stronger result is known for embeddings to $\C^n$: Every Stein manifold of dimension $d>1$ admits a proper holomorphic embedding into $\C^n$ with $n=\left[\frac{3d}{2}\right] +1$ (Eliashberg and Gromov \cite{Eliashberg-Gromov1992}, Sch\"urmann \cite{Schurmann}.) We do not know whether this result can be generalized to the target manifolds considered in this paper.

%
%
%   EXAMPLES
%
%
\begin{example}
We illustrate the scope of Theorems \ref{th:main} and \ref{th:main2} by collecting known examples of Stein manifolds with the  density property or the volume density property. 
\begin{itemize}
\item The complex Euclidean space $\C^n$ for any $n \geq 1$ satisfies the volume density property with respect to the volume form $dz_1 \wedge \cdots \wedge dz_n$ (see \cite{Andersen}).
\item $\C^n$ for any $n \geq 2$ satisfies the density property  (see \cite{Andersen-Lempert}).
\item The Stein manifold $(\C^*)^n$ for $n \geq 1$ satisfies the volume density property with respect to the volume form $\frac{dz_1}{z_1}\wedge\cdots\wedge \frac{dz_n}{z_n}$. (Here $\C^* =\C\setminus\{0\}$.) It is not known whether $(\C^*)^n$ satisfies the density property when $n>1$.
\item For any Stein Lie group $G$ with an invariant Haar form $\omega$, $G \times \C$ has the density property, and the volume density property with respect to $\omega \wedge dz$ (see \cite{Varolin2001}).
\item If $G$ is a linear algebraic group and $H \subset G$ is a closed proper reductive subgroup, then the homogeneous space $X=G/H$ is a Stein manifold with the density property, except when $X = \C, (\C^\ast)^n$, or a $\Q$-homology plane with fundamental group $\Z_2$ (see \cite[Theorem 6]{Donzelli}).
\item In particular, a linear algebraic group with connected components different from $\C$ or $(\C^\ast)^n$ has the density property (see \cite[Theorem 3]{Kaliman-Kut2008}).
\item If $p: \C^n \to \C$ is a holomorphic function with smooth reduced zero fibre, then the Stein manifold $X=\{(x,y,z)\in \C \times \C \times \C^n  \colon x y = p(z)\}$ 
has the density property (see \cite{Kaliman-Kut2008-1}).
\item A Cartesian product $X_1 \times X_2$ of two Stein manifolds $X_1, X_2$ with the density property also has the density property. A product $X_1 \times X_2$ of two Stein manifolds $(X_1, \omega_1), (X_2, \omega_2)$ with the volume density property satisfies the volume density property with respect to the volume form $\omega_1 \wedge \omega_2$ 
(see \cite[Theorem 1]{Kaliman-Kut2011}).
\item Given a reductive subgroup $R$ of any linear algebraic group $G$ 
such that the homogeneous space $X=G/R$ has a left invariant volume form 
$\omega$, the space $(X,\omega)$ has the volume density property
(see \cite[Corollary 6.2]{KK2012}).
\end{itemize}
\end{example}

%
%
%  THE MAIN LEMMA
%
%
\section{The main lemma}
\label{sec:mainlemma}
In this section we develop the key  analytic  ingredients that will be used in the proof of Theorem \ref{th:main}, namely Lemma \ref{lem:main} and Proposition \ref{prop:noncritical}.  

Recall that a compact set $K$ in a complex manifold $S$  is said to be a {\em Stein compact} if $K$ admits a basis of open Stein neighborhoods in $S$. If $K\subset A$ are compacts in $S$, we say that $K$ is $\cO(A)$-convex if there is an open set $U\subset S$ containing $A$ such that $K$ is $\cO(U)$-convex. (The latter notion was defined just below the statement of Theorem \ref{th:main}.) 

We recall the notion of a (special) Cartan pair and of a convex bump, adjusting slightly Definition 5.10.2 in \cite[p.\ 218]{Forstneric2011}. (See also Fig.\ 5.2 in \cite[p.\ 219]{Forstneric2011}.)

\begin{definition}
\label{def:CP}
A pair of compact sets $(A,B)$ in a complex manifold $S$ is said to be a {\em Cartan pair} if 
\begin{itemize}
\item[\rm(i)]  the sets $A$, $B$, $D:=A\cup B$ and $C:=A\cap B$ are Stein compacta, and
\item[\rm(ii)]  $A,B$ are {\em separated} in the sense that
$\overline{A\bs B}\cap \overline{B\bs A} =\emptyset$.
\end{itemize}
A Cartan pair $(A,B)$ is said to be {\em special}, and $B$ is said to be a {\em convex bump on $A$}, if in addition the following properties hold:
\begin{itemize}
\item[\rm (iii)] $A$ and $D=A\cup B$ are compact strongly pseudoconvex domains, and 
\item[\rm (iv)] there is a holomorphic coordinate system on a neighborhood of $B$ in $S$ 
in which the sets $B$ and $C=A\cap B$ are strongly convex.
\end{itemize}
\end{definition}

In the cited definition in \cite[p.\ 218]{Forstneric2011} the sets $B$ and $C$ were not explicitly required to be convex along $bC\cap \mathring A = bB\cap \mathring A$ (it was merely asked that in local coordinates near $B$ the boundaries of $A$ and $A\cup B$ be strongly convex). This addition is trivially satisfied in the context of that definition by a suitable choice of 
the set $B$.

\begin{lemma}
\label{lem:main}
Assume that $S$ is a Stein manifold of dimension $d$ and $X$ is a Stein manifold of dimension $n$, where $2d+1\le n$. Let $D=A\cup B$ be a compact strongly pseudoconvex domain in $S$ such that $(A,B)$ is a special Cartan pair, and $B$ is a compact convex bump attached to $A$ along the convex set $C=A \cap B$ (see Def.\ \ref{def:CP}). Assume that 
\begin{itemize}
\item[\rm (a)] $L$ is a compact $\cO(X)$-convex set in $X$, 
\item[\rm (b)] $K$ is a compact set contained in $\mathring A\setminus C$ such that $K\cup C$ is $\cO(A)$-convex,  
\item[\rm (c)] $W\subset S$ is an open set containing $A$, and 
\item[\rm (d)] $f\colon W\hra X$ is an injective holomorphic immersion such that $f^{-1}(L)\subset \mathring K$.
\end{itemize}
If $X$ has the density property or the volume density property, then it is possible to approximate $f$ as closely as desired, uniformly on $A$, by an injective holomorphic immersion $\tilde f\colon \wt W\to X$ on a neighborhood $\wt W$ of 
$D=A\cup B$ such that $\tilde f^{-1}(L)\subset \mathring K$. 
\end{lemma}

\begin{proof}
We first consider the case when $X$ satisfies the density property; the necessary changes in the case of volume density property will be explained at the end. 

Replacing the Stein manifold $S$ with a suitable Stein neighborhood of the compact strongly pseudoconvex domain $D=A\cup B$, we may assume that the sets $A$ and $D$ are $\cO(S)$-convex. It then follows from condition (b) in the lemma 
that the sets $C$ and $K\cup  C$ are also $\cO(S)$-convex. 

Pick a smoothly bounded strongly pseudoconvex Runge domain $W_0$ in $S$ such that $A\subset W_0 \Subset W$ and $A$ is $\cO(W_0)$-convex.  By Theorem 1.1 in \cite{DF2010} we can approximate $f$ uniformly on $A$ by a proper holomorphic embedding $g\colon W_0\hra X$ such that $g^{-1}(L) \subset \mathring K$. To see this, pick a strongly plurisubharmonic exhaustion function $\sigma\colon X\to \R$ such that $L\subset \{\sigma<0\}$ and $\sigma >0$ on the compact set $f(\overline {W_0 \setminus K})$. Given a number $\epsilon>0$, the cited result lets us approximate $f$ uniformly on $A$ by a proper holomorphic map $g\colon W_0\to X$ satisfying $\sigma(g(z)) > \sigma(f(z))-\epsilon$ for all $z\in W_0$. 
Choosing $\epsilon>0$ small enough obtain $g^{-1}(L) \subset \mathring K$ as claimed. The fact that $g$ can be chosen an embedding follows by a general position argument since $2d+1\le n$.

The image $\Sigma:= g(W_0) \subset X$ is a closed complex submanifold of $X$. 
Since $K\cup C$ is $\cO(A)$-convex and $A$ is $\cO(W_0)$-convex, 
$K\cup C$ is also $\cO(W_0)$-convex, and hence the image $g(K\cup C)$ is $\cO(\Sigma)$-convex.
From $g^{-1}(L) \subset \mathring K$ we also get  $L\cap \Sigma \subset g(K) \subset g(K\cup C)$. 
Lemma 6.5 in \cite{Forstneric1999} then shows that the set $L\cup g(K\cup C)$ is also 
$\cO(X)$-convex. Replacing $f$ by $g$ and $W$ by $W_0$ we may assume that $f$ satisfies these properties. 

Set $L'=L\cup f(K)$; then $L'\cap f(C)=\emptyset$ and the sets $L'$, $f(C)$ and $L'\cup f(C)$ are $\cO(X)$-convex. 
Pick a compact set $P\subset X\setminus L'$, containing $f(C)$ in its interior, such that $L' \cup P$ is also $\cO(X)$-convex. 
% We may choose $P$ to be a small analytic polyhedron around $f(C)$.

The hypotheses on the pair $(A,B)$ imply that there exists a holomorphic coordinate system $z=(z_1,\ldots, z_d)\colon V_0\to\C^d$ on a neighborhood $V_0 \subset S$ of $B$ such that, in these coordinates, the compact sets $B$ and $C=A\cap B$ are geometrically convex. Recall that the embedding $f$ is defined on a neighborhood $W$ of $A$. 

Choose open convex neighborhoods $U,V\subset S$ of the sets $C$ and $B$, respectively, such that  $U\subset V\cap W$ and $V\subset V_0$. (More precisely, the sets $z(U) \subset z(V) \subset \C^d$ are assumed to be convex.) We can find an isotopy $r_t\colon V\to V$ of injective holomorphic self-maps, depending smoothly on the parameter $t\in [0,1]$, such that 
\begin{enumerate}
\item $r_0$ is the identity map on $V$, 
\item $r_t(U)\subset U$ for all $t\in [0,1]$, and 
\item $r_1(V)\subset U$. 
\end{enumerate}
In the coordinates $z$ on $V_0$ we can simply choose $r_t$ to be a family of linear contractions towards a point in $U$.

Since the set $U$ is convex, the normal bundle of the embedding $f\colon W\hra X$ 
is holomorphically trivial over $U$ by the Oka-Grauert principle (cf.\ \cite[\S 5.3]{Forstneric2011}). 
Hence, after shrinking $W$ around $A$ and $U$ around $C$ if necessary, there is a holomorphic map $F\colon W\times \D^{n-d} \to X$, where $\D^{n-d}$ denotes the polydisc in $\C^{n-d}$, such that $F$ is injective holomorphic on $U\times \D^{n-d}$ and $F(z,0)=f(z)$ for all $z\in W$. By a further shrinking of the neighborhood $U\supset C$ and rescaling in the fiber variable $w\in \D^{n-d}$ we may also assume that the Stein domain 
\[
	\Omega:=F(U\times \D^{n-d})\subset P\subset X\setminus L'
\]
is Runge in $\mathring P$ and its closure $\overline \Omega$ is $\cO(P)$-convex. 
Since $L'\cup P$ is $\cO(X)$-convex, it follows that $L' \cup \overline{\Omega}$ is also $\cO(X)$-convex. 
Hence there is a Stein neighborhood $\Omega'\subset X$ of $L'$ such that 
$\overline \Omega \cap \overline \Omega'=\emptyset$ and the union 
$\Omega_0:=\Omega\cup\Omega'$ is a Stein Runge domain in $X$. 

Consider the  isotopy of biholomorphic maps $\phi_t\colon V\times \D^{n-d}\to V\times \D^{n-d}$ given by 
\begin{equation}
\label{eq:phit}
	\phi_t(z,w)= (r_t(z),w),\quad z\in V,\ w\in \D^{n-d},\ t\in [0,1].  
\end{equation}
We define a smooth isotopy of injective holomorphic maps $\psi_t\colon \Omega_0 \to X$ ($t\in [0,1]$) by \begin{equation}
\label{eq:conj}
		\psi_t = F\circ \phi_t \circ F^{-1}\quad \text{on}\ \Omega; \qquad \psi_t=\mathrm{Id}\quad \text{on}\ \Omega'.
\end{equation} 
The map $\psi_t$ is defined on $\Omega$ since $r_t(U)\subset U$ for all $t\in [0,1]$. 
Note that $\psi_0$ is the identity on $\Omega_0$ and the domain $\psi_t(\Omega_0)$ is Runge in $X$ for all $t\in [0,1]$. 
By the Anders\'en-Lempert-Forstneri\v c-Rosay theorem \cite[Theorem 4.10.6, p.\ 132]{Forstneric2011} we can approximate the map $\psi_1= F\circ \phi_1 \circ F^{-1}\colon \Omega_0 \to X$ 
uniformly on compacta in $\Omega_0$ by holomorphic automorphisms $\Psi\in \Aut X$. 
Fix such $\Psi$ and consider the injective holomorphic map 
\[
	G = \Psi^{-1}\circ F\circ \phi_1 \colon V\times \D^{n-d}\to X.
\]
Observe that $G$ is indeed defined on $V\times \D^{n-d}$ since we have $\phi_1(z,w)= (r_1(z),w)$ 
and $r_1(V)\subset U$ (see condition (3) above), so $\phi_1(V\times \D^{n-d}) \subset U\times \D^{n-d}$.

Since $\psi_1$ equals the identity map on $\Omega' \supset L'$ by  (\ref{eq:conj}), 
$\Psi$ can be chosen  to approximate the identity as closely as desired 
on a neighborhood of $L'$, so we may assume that $G(V\times \D^{n-d}) \subset X\setminus L'$. 
From the first equation in (\ref{eq:conj}) we see that
\[
	G=\Psi^{-1}\circ \psi_1 \circ F \quad\text{on} \ U\times \D^{n-d}. 
\]
Since $\Psi^{-1} \circ \psi_1$ is close to the identity map on $F(U\times \D^{n-d})$ by the choice of $\Psi_1$, 
$G$ is close to $F$ on $U\times \D^{n-d}$. More precisely, the above argument shows that for any 
given compact subset $M$ of  $U\times \D^{n-d}$ we can choose the automorphism 
$\Psi\in \Aut(X)$ such that the associated map $G$ is as close as desired to $F$ on $M$.

Assuming as we may that the approximation of $F$ by $G$ is close enough, and
after shrinking their domains slightly, we can glue the holomorphic maps 
$F\colon W\times \D^{n-d}\to X$ and $G\colon V\times \D^{n-d}\to X$ 
into a holomorphic map $\wt F\colon (A\cup B)  \times \rho\D^{n-d}\to X$ for some $0<\rho <1$ such that $\wt F$ is close to $F$ on $A\times \rho \D^{n-d}$, and is close to $G$ on $B\times \rho \D^{n-d}$. (We apply the gluing lemma furnished by \cite[Theorem 4.1]{FF:Acta}; see also Theorem 8.7.2 in \cite[p.\ 359]{Forstneric2011}.) 
The holomorphic map $\tilde f:= \wt F(\cdotp,0) \colon A\cup B =D\to X$ then satisfies the conclusion of
Lemma \ref{lem:main}, except that it need not be an injective embedding; this can be achieved by a 
small perturbation since $2d+1\le n$. If all approximations were close enough, 
then the intersection set of $\tilde f$ with the compact set  $L\subset X$ is close to the 
intersection set of $f$ with $L$, so we have $\tilde f^{-1}(L)\subset \mathring K$. 

This proves Lemma \ref{lem:main} when $X$ satisfies the density property.

Assume now that $X$ satisfies the volume density property with respect to a holomorphic volume form $\omega$. Choose a sufficiently small number $\rho>0$ and change the definition of the isotopy $\phi_t\colon V\times \rho\D^{n-d}\to V\times \D^{n-d}$  in (\ref{eq:phit}) by taking
\[
	\phi_t(z,w)=\left(r_t(z),g_t(z,w)\right),\quad z\in V,\ w\in \rho\D^{n-d},
\] 
where the holomorphic map $g_t(z,w)$ is chosen such that $g_0(z,w)=w$, $g_t(z,0)=0$ for all $z\in V$ and $t\in [0,1]$, and $\phi_t$ preserves the volume form $F^*\omega$ on $V\times \D^{n-d}$. (See \cite{AndristWold} for the details.) 
The conjugate isotopy $\psi_t$ given by  (\ref{eq:conj}) then preserves the volume form $\omega$, in the sense that $\psi_t^*\omega=\omega$ for all $t\in[0,1]$. Since the set $V\times \rho \D^{n-d}$
(and hence its image $F(V\times \rho \D^{n-d})$) is contractible, all its cohomology groups vanish. Hence the volume-preserving version of the Anders\'en-Lempert-Forstneri\v c-Rosay theorem applies and shows that $\psi_1$ can be approximated by $\omega$-preserving automorphisms $\Psi$ of $X$ which are close to the identity map on $\Omega'$. 
(A word of explanation is in order. We only need the vanishing of the group $H^{n-1}(\cdotp,\C)$ to approximate the infinitesimal generator of the isotopy on small time intervals by flows of globally defined holomorphic vector fields on $X$ with vanishing $\omega$-divergence; see the proof of Theorem 4.9.2 in \cite[p.\ 125]{Forstneric2011} for the case when $\omega$ is the standard volume form $dz_1\wedge\cdots\wedge dz_n$ on $X=\C^n$. 
Although this cohomology vanishing condition need not hold for the domain $\Omega'$, 
this is irrelevant since we are approximating the constant isotopy $\psi_t=\mathrm{Id}$ there. 
Now apply Proposition 4.10.4 in \cite[p.\ 132]{Forstneric2011} to obtain the 
approximation by $\omega$-preserving automorphisms of $X$. This argument can also
be found in \cite[Theorem 2]{Kaliman-Kut2011}.)  The rest of the proof is exactly as in the 
case of a manifold $X$ with the density property. 
\end{proof}

By using Lemma \ref{lem:main} we shall now obtain the  {\em noncritical case} in the proof of Theorem \ref{th:main}. Let $A\subset A'$ be compact strongly pseudoconvex domains in a Stein manifold $S$. We say that $A'$ is a {\em noncritical strongly pseudoconvex extension} of $A$ (cf.\ \cite[p.\ 218]{Forstneric2011}) if there exist a strongly plurisubharmonic function $\rho$ without critical points on a neighborhood $U$ of $\overline{A'\setminus A}$ and a pair of real numbers $c<c'$ satisfying
\begin{equation}
\label{eq:noncritical}
	U\cap A= \{z\in U\colon \rho(z)\le c\},\quad U\cap A' = \{z\in U\colon \rho(z)\le c'\}.
\end{equation}

\begin{proposition}
\label{prop:noncritical}
Assume that $A\subset A'$ is a noncritical strongly pseudoconvex extension in a Stein manifold $S$. Let $X$ be a Stein manifold with the density property (or the volume density property) satisfying $2\dim S +1\le \dim X$, and let $L\subset X$ be a compact $\cO(X)$-convex set. Given an injective holomorphic immersion $f\colon W \hra X$ on an open set $W \supset A$ such that $f^{-1}(L) \subset \mathring A$, we can approximate $f$ as close as desired uniformly on $A$ by an injective holomorphic immersion $f'\colon W' \hra X$, defined on a small open neighborhood $W' \subset S$ of $A'$, such that $f'(W'\setminus \mathring A)\subset X\setminus L$.
\end{proposition}

\begin{proof}
Set $d=\dim S$ and $n=\dim X$, so we have $2d+1\le n$. Replacing $S$ by a suitable Stein neighborhood of the compact set $A'$ we may assume that the sets $A$ and $A'$ are $\cO(S)$-convex. Pick a compact set $K\subset \mathring A$ such that $K$ is $\cO(S)$-convex and $f^{-1}(L)\subset \mathring K$. (With $\rho$ as in (\ref{eq:noncritical}) we can simply take $K=\{z\in A\colon \rho(z)\le c_0\}$ for some constant $c_0<c$ close to $c$.) Choose a finite open cover $\{U_1,\dots, U_l\}$ of the compact set $\overline{A'\setminus A}$ such that for each $j=1,\ldots, l$ there is a biholomorphic map $\theta_j\colon U_j\to \B^d \subset \C^d$ onto  the unit ball in $\C^d$ and the set $K\cup \overline U_j$ is $\cO(S)$-convex. The latter condition can be satisfied by choosing the sets $U_j$ small enough.

By Lemma 5.10.3 in \cite[p.\ 218]{Forstneric2011} there exist compact strongly pseudoconvex domains $A=A_0 \subset A_1 \subset \cdots \subset A_m=A'$ such that for every $k=0,1,\ldots,m-1$ we have $A_{k+1}=A_k\cup B_k$, where $B_k$ is a convex bump on $A_k$ and $(A_k,B_k)$ is a special Cartan pair (see Def.\ \ref{def:CP}); furthermore, each set $B_i$ is contained in one of the sets $U_{j(i)}$, and the compact sets $B_i$ and $C_i=A_i\cap B_i$ are geometrically convex with respect to the holomorphic coordinates $\theta_{j(i)}\colon U_{j(i)} \to\B^d$. In particular, $B_i$ and $C_i$ are $\cO(U_{j(i)})$-convex. Since $K\cup \overline U_j$ is $\cO(S)$-convex for each $j=1,\ldots, l$, it follows that $K\cup B_i$ and $K\cup C_i$ are also $\cO(S)$-convex (and hence $\cO(A_{i+1})$-convex) for every $i=0,1,\ldots, m-1$. 

By Lemma \ref{lem:main} we inductively find injective holomorphic immersions $f_i\colon W_i\hra X$ for $i=0,\ldots,m$, with $f_0=f$, such that for every $i=1,\ldots, m$ the map $f_i$ is defined on a small open neighborhood $W_i$ of $A_i$, it approximates the previous map $f_{i-1}$ as closely as desired uniformly on $A_{i-1}$, and it satisfies $f_i^{-1}(L) \subset \mathring K$. All the hypotheses in Lemma \ref{lem:main} are clearly satisfied at each step of the induction. If the aproximation is sufficiently close at every step, then the final map $f'=f_m\colon W_m \to X$, which is defined on a neighborhood $W' = W_m$ of $A'$, satisfies the conclusion of the proposition. 
\end{proof}

%%%%%%%%%%%%%%%%%%%%%%%%%%%%%%%%%%%%%%%%%%%
%
%  THE PROOF
%
%%%%%%%%%%%%%%%%%%%%%%%%%%%%%%%%%%%%%%%%%%%

\section{Proof of Theorems \ref{th:main} and \ref{th:main2}.}
\label{sec:proof}
The proof amounts to an inductive application of Proposition \ref{prop:noncritical}, with an additional argument at critical points of an exhaustion function on the Stein manifold $S$. The procedure is similar to the one used in Oka theory (cf.\ Chapter 5 in \cite{Forstneric2011}), but with additional ingredients to ensure properness of the limit map. 

We shall focus on Theorem \ref{th:main}; it will be clear that the same construction gives the more precise statement in Theorem \ref{th:main2}. We begin with the case when $S'=\emptyset$, that is, without the interpolation condition. 

By general position we may assume that the initial map $f_0=f$ is an injective holomorphic immersion on an open set $U_0\subset S$ containing the given compact set $K$ in Theorem \ref{th:main}. (We may assume that $K$ is nonempty since we can deform $f$ so that it becomes an injective holomorphic immersion on a small open set in $S$.) 

Since $K$ is $\cO(S)$-convex, there is a smooth strongly plurisubharmonic Morse exhaustion function $\rho\colon S\to\R$ such that $\rho<0$ on $K$, $\rho>0$ on $S\setminus U_0$, and $0$ is a regular value of $\rho$. Let $p_1,p_2,\ldots$ be the critical points of $\rho$ in $\{\rho>0\}$, ordered so that $0<\rho(p_1)<\rho(p_2)<\cdots$. Choose an increasing sequence of numbers $c_0=0<c_1<c_2<\cdots$ with $\lim_{j\to\infty} c_j=+\infty$ such that every $c_j$ is a regular value of $\rho$ and we have $c_{2j-1}< \rho(p_j) < c_{2j}$ for $j=1,2,\ldots$. (If there are only finitely many critical points $p_j$ then we choose the rest of the sequence $c_j$ arbitrarily, subject to the condition $\lim_{j\to\infty} c_j=+\infty$.) Furthermore, the numbers $c_{2j-1}$ and $c_{2j}$ are chosen close to $\rho(p_j)$  (this will be specified later). Set 
\[
	K_j=\{z\in S\colon \rho(z) \le c_j\},\qquad j=0,1,2\ldots; 
\]
then $K\subset K_0\subset K_1\subset  \cdots \subset \bigcup_{j=0}^\infty K_j=S$, every $K_j$ is $\cO(S)$-convex, and $K_{j-1}\subset \mathring K_{j}$ holds for all $j=1,2,\ldots$. 

We also pick an exhaustion $L_1\subset L_2\subset \cdots \subset \bigcup_{j=1}^\infty L_j=X$ of $X$ by compact $\cO(X)$-convex sets. For convenience we assume that $L_j=\{\sigma \le j\}$ for a smooth strongly plurisubharmonic Morse exhaustion function $\sigma \colon X\to \R$, chosen such that every integer $j\in \N$ is a regular value of $\sigma$. In the situation of Theorem \ref{th:main2} we can assume that $L_1=L$ is the given $\cO(X)$-convex set such that condition (c) in Theorem \ref{th:main2} holds. 
 
Set $f_0=f$ and pick a number $\epsilon_0>0$. Choose a distance function $\dist$ on $X$ induced by a complete Riemannian metric. We shall find a sequence of continuous maps $f_j\colon S\to X$ and a sequence of numbers $\epsilon_j>0$ such that the following conditions hold for every $j=1,2,\ldots$:
\begin{itemize}
\item[\rm (a)]  $f_j$ is an injective holomorphic immersion on an open neighborhood $W_j$ of $K_j$,  
\item[\rm (b)]  $\sup_{z\in K_{j-1}} \dist (f_j(z),f_{j-1}(z)) <\epsilon_{j-1}$, 
\item[\rm (c)]  $f_j(\overline{K_j\setminus K_{j-1}}) \subset X\setminus L_j$, 
\item[\rm (d)]  there is a homotopy $h_{j,t}\colon S\to X$ $(t\in [0,1])$, with $h_{j,0}=f_{j-1}$ and $h_{j,1}=f_{j}$, such that every map $h_{j,t}$ in the family is holomorphic on a neighborhood of $K_{j-1}$ and we have the estimate  
\begin{equation}
\label{eq:hjt}
	\sup_{z\in K_{j-1}} \dist (h_{j,t}(z),f_{j-1}(z)) <\epsilon_{j-1}, \quad t\in [0,1],
\end{equation}
\item[\rm (e)]  $0<\epsilon_j<\epsilon_{j-1}/2$, and
\item[\rm (f)]  every holomorphic map $F\colon S\to X$ satisfying $\sup_{z\in K_{j}} \dist (F(z),f_j(z)) < 2\epsilon_j$ is an injective immersion on $K_{j-1}$.
\end{itemize}

Let us explain the construction. Assuming that the maps $f_0,\ldots, f_{j-1}$ and the corresponding numbers $\epsilon_0,\ldots,\epsilon_{j-1}$ with the required properties have already been found (this is true for $j=1$), we must explain the contruction of the next map $f_j$ and the choice of the number $\epsilon_j$. There are two distinct cases to consider: the {\em noncritical case} when $K_j\setminus K_{j-1}$ does not contain any critical point of $\rho$ (in our notation this happens for odd values of $j$), and the {\em critical case} when the set $K_{j}\setminus K_{j-1}$ contains a critical point of $\rho$ (this happens for even values of $j$). We shall explain in detail how to get the maps $f_1$ and $f_2$; all the subsequent steps are analogous to one of these two cases.

The initial map $f_0\colon S\to X$ is holomorphic on the open set $U_0\supset K_0$. 
Choose a smoothly bounded strongly pseudoconvex open domain $D_0$ in $S$ with $K_0\subset D_0\Subset U_0$. 
(We may simply take $D_0=\{\rho<c\}$ for a sufficiently small $c>0$.)
By \cite[Theorem 1.1]{DF2010} there is a holomorphic map $g\colon \overline{D}_0\to X$ such that 
\[
	g(bD_0) \subset X\setminus L_1 \quad\text{and}\quad 
	\sup_{z\in K_0} \dist(g(z),f_0(z))< \frac{\epsilon_0}{2}. 
\]
(In the situation of Theorem \ref{th:main2} we can choose $K_0$ and $D_0$ such that $f_0(\overline{D_0\setminus K_0})\subset X\setminus L_1$, so the above condition holds with $g=f_0$.) Furthermore, the map $g$ can be chosen such that there is a homotopy from $f_0$ to $g$, consisting of holomorphic maps on $D_0$ and satisfying the approximation condition in (d) for $j=1$ (that is, on the set $K_0$), with $\epsilon_0$ replaced by $\epsilon_0/2$. (Such a homotopy exists whenever the approximation of $f_0$ by $g$ is sufficiently close on $K_0$.) 
The map $g$, and the homotopy from $f_0$ to $g$, can be extended continuously to all of $S$, without changing them on a small neighborhood of $K_0$, by using a cut-off function in the parameter of the homotopy. 

Since the set $K_1=\{\rho\le c_1\}$ is a noncritical strongly pseudoconvex extension of the set $K_0=\{\rho\le 0\}$, Proposition \ref{prop:noncritical} above furnishes an injective holomorphic immersion $f_1\colon W'_1\to X$ on an open set $W'_1\supset K_1$ in $S$ such that 
\[
	f_1(\overline{K_1\setminus K_0}) \subset X\setminus L_1 \quad \text{and}\quad  
	\sup_{z\in K_0} \dist(f_1(z),g(z))< \frac{\epsilon_0}{2}.
\]
Furthermore, we can ensure that there exists a homotopy of holomorphic maps from $g$ to $f_1$ on a neighborhood of $K_0$ satisfying the estimate (\ref{eq:hjt}) for $j=1$, with $\epsilon_0$ replaced by $\epsilon_0/2$ and $f_0$ replaced by $g$. As before, we extend the map $f_1$ and the homotopy continuously to all of $S$ without changing their values on a smaller neighborhood $W_1$ of the set $K_1$. By combining these two homotopies (the first one from $f_0$ to $g$, and the second one  from $g$ to $f_1$) we get a homotopy $h_{1,t}$ from $f_0$ to $f_1$, consisting of maps that are holomorphic on a neighborhood of $K_0$ and satisfy the estimate (\ref{eq:hjt})  for $j=1$. Clearly the map $f_1$ satisfies properties (a)--(d) for $j=1$. Now pick a number $\epsilon_1$ satisfying $0<\epsilon_1 <\epsilon_0/2$ such that Condition (f) is satisfied for $j=1$ (this holds whenever $\epsilon_1>0$ is sufficiently small). 
This completes the first step of the induction. 

In the next step we must find the map $f_2\colon S\to X$. This {\em critical case} is accomplished in finitely many
substeps which we now describe; for further details we refer to pp.\ 222--223 in \cite[\S 5.11]{Forstneric2011}.  

We may assume that the number $c_1$ is so close to the critical value $\rho(p_1)$ that we can work in a coordinate neighborhood of the point $p_1$ in $S$ in which $\rho$ assumes the quadratic normal form furnished by \cite[Lemma 3.9.1, p.\ 88]{Forstneric2011}. We attach to the strongly pseudoconvex domain $K_1=\{\rho\le c_1\}$ the local stable manifold $E$ of the critical point $p_1$. In the local holomorphic coordinate in which $\rho$ assumes the normal form, $E$ is a linear totally real disc of dimension $k$ that equals the Morse index of $p$ (indeed, $E$ is a closed ball in $\R^k\subset \R^d\subset \C^d$), it is contained in $\{\rho > c_1\}=X\setminus K_1$ except for $bE$, and is attached to $K_1$ along a legendrian  (complex tangential) sphere $S^{k-1}\cong bE$ contained in the strongly pseudoconvex hypersurface $bK_1= \{\rho=c_1\}$.  

Choose a number $c'_1$ with $c_1<c'_1<\rho(p_1)$ and  $c'_1$ very close to $\rho(p_1)$. (The latter condition will be made precise in the sequel.) By the noncritical case explained above, we can assume that $f_1$ is holomorphic on a neighborhood of the set $D_1:=\{\rho\le c'_1\}$ and it satisfies $f_1(\overline{D_1 \setminus K_0}) \subset X\setminus L_1$. We may assume that the properties (a)--(f) still hold for the new $f_1$. Let $E'=E\cap \{\rho\ge c'_1\}$; this is again a totally real $k$-disc attached with its boundary sphere $bE'\cong S^{k-1}$ to the domain $D_1$.

Consider the continuous map $f_1|_{E'} \colon E'\to X$. Note that $f_1(bE')\subset X\setminus L_1$. We claim that $f_1|_{E'}$ can be homotopically deformed to a map with values in $X\setminus L_1$, keeping the homotopy fixed near $bE'$. To see this, observe that $X$ is obtained from $X\setminus \mathring L_1$  by attaching to the latter set handles of index at least $n$. (This is because the critical points of the function $-\sigma \colon X\to \R$ have Morse indices at least $n$, and $X\setminus \mathring L_1=\{-\sigma \le -1\}$.) It follows that the relative homotopy groups $\pi_l(X,X\setminus L_1)$ vanish for $l=0,\ldots,n-1$. Since $k\le d<n$, we see that any map $(E',bE')\mapsto (X,X\setminus L_1)$ is homotopic to a map with values in $X\setminus L_1$, so the claim follows. We can extend this homotopy to all of $X$ without changing it on $D_1$. We still denote the new map by $f_1$. By the construction we see that $f_1$ maps the compact set $(D_1\cup E')\setminus \mathring K_0$ to $X\setminus L_1$.  

We now apply the Mergelyan's theorem in \cite[Theorem 3.7.2, p.\ 81]{Forstneric2011} to make $f_1$ holomorphic on a neighborhood $\Omega\subset S$ of the set $D'_1\cup E'$. To simplify the notation we still denote the new map by $f_1$ and assume that properties (a)--(f) hold for $j=1$.   

Assuming that the number $c'_1$ is close enough to $\rho(p_1)$, Lemma 3.10.1 in \cite[p.\ 92]{Forstneric2011} gives a smooth strongly plurisubharmonic function $\tau$ on $S$ and constants $0<c<c'$ satisfying the following properties: 
\begin{itemize}
\item $K_1\cup E \subset \{\tau <c\} \subset \Omega$,  
\item $K_2=\{\rho\le c_2\} \subset \{\tau <c'\}$, and 
\item $\tau$ has no critical values in the interval $[c,c']$.
\end{itemize}

Applying the noncritical case explained above with $\tau$ and $f_1$, 
we find a map $f_2$ and a homotopy from $f_1$ to $f_2$ satisfying properties (a)--(d) for $j=2$. 
Pick a number $\epsilon_2>0$ so that conditions (e) and (f) hold for $j=2$. This completes the construction of $f_2$.

The subsequent steps in the induction are exactly the same as in one of these two cases, depending on the parity of the index $j$, so the induction proceeds. 

\smallskip
\noindent {\em Conclusion of the proof of Theorem \ref{th:main} (assuming $S'=\emptyset$):} Conditions (a) and (e) ensure that the sequence $f_j$ converges uniformly on compacta in $S$ to a holomorphic map $F=\lim_{j\to\infty} f_j\colon S\to X$. Condition (d) shows that the sequence of homotopies $h_{j,t}$ also converges uniformly on compacta to a homotopy $h_t\colon S\to X$ $(t\in [0,1])$ from $h_0=f_0=f$ to $h_1=F$. Properties (b) and (e) imply  
\begin{equation}
\label{eq:Ffj}
	\sup_{z\in K_j} \dist(F(z),f_j(z)) < 2\epsilon_j\quad \text{for}\ j=0,1,2,\ldots. 
\end{equation}
In particular, we have $\sup_{z\in K} \dist(F(z),f(z)) < 2\epsilon_0$. The estimate (\ref{eq:Ffj}), together with property (f) of the sequence $\epsilon_j$, shows that $F\colon S\to X$ is an injective holomorphic immersion on $K_{j-1}$. Since this holds for every $j$, it follows that $F$ is an injective immersion on all of $S$. Finally, condition (c) together with (\ref{eq:Ffj}) shows that $F$ is proper.

This completes the proof of Theorems \ref{th:main} and \ref{th:main2} in the case $S'=\emptyset$.

\smallskip
\noindent {\em The general case:} The restricted map $f_{S'}\colon S'\to X$ is by the assumption a proper holomorphic embedding. The construction of a proper holomorphic embedding $F\colon S\hookrightarrow X$ satisfying 
$F|_{S'}=f|_{S'}$ requires a minor modification of the induction scheme. One can proceed as in \cite{ForstnericRitter} (see in particular \S 4 in the cited paper, Case 2, pp.\ 11--12 in the arXiv preprint), except that the individual steps in the induction are accomplished as described in this paper. 
An important point is that none of the bumps which are used in this construction intersect the subvariety $S'$, 
and hence we can perform the gluing at every step so as to satisfy the required interpolation condition on $S'$.

\begin{remark}
\label{rem:anotherapproach} 
Assuming that $X$ has the density property, one can also prove 
Theorem \ref{th:main2} by following the approach in \cite{ForstnericRitter}.
The main difference with respect to this paper lies in the proof of Lemma \ref{lem:main}.
When attaching a bump $B$ to a set $A$, with an embedding $f:A\to X$, we may assume 
by \cite[Theorem 1.1]{DF2010} that $f$ is proper holomorphic on some neighborhood $W\subset S$ of $A$, 
so  $f(C)$ is an embedded holomorphically contractible set such that $L\cup f(C)$ is 
holomorphically convex in $X$.  (Here $C=A\cap B$.) Thus we can pick a point $x\in f(C)$ and construct a 
sequence $\phi_j$ of holomorphic automorphisms of $X$ such that $f(C)$ is contained in the basin of attraction
$\Omega\subset X$ of  $(\phi_j)_{j\in\N}$, but the set $L$ does not intersect $\Omega$.
Assuming as we may that the sequence $\phi_j$ satisfies the uniform attraction condition 
$a|z|\leq |\phi_j(z)|\leq b|z|$ with $0<b^2<a<b< 1$ on a ball centered at the point $x$
(in some local holomorphic coordinates $z$ around $x$ with $z(x)=0$), 
the basin $\Omega$ is biholomorphic to $\C^n$. (See \cite[Theorem 9.1]{RR1988} for the
special case of iteration of an automorphism, and \cite{Wold2005} for the general case
of a random iteration.)  Hence one can approximate and glue maps just as was done  in \cite{ForstnericRitter}. 
One can find a sequence $\phi_j$ with these properties using the 
Anders\'en-Lempert-Forstneri\v c-Rosay theorem. This approach does not seem to work
by using the volume density property of $X$. 
\end{remark}

\smallskip
\textit{Acknowledgements.}
F.\ Forstneri\v c is supported by research program P1-0291 and research grant J1-5432 from ARRS, Republic of Slovenia. T.\ Ritter is supported by Australian Research Council grant DP120104110. E.\ F. Wold is supported by grant  NFR-209751/F20 from the Norwegian Research Council.

\bibliographystyle{amsplain}

\end{document}